\documentclass[reqno, 12pt]{amsart}

\usepackage[latin5]{inputenc}
\usepackage{graphicx}
\usepackage{epsfig}
\usepackage{caption}
\usepackage{subcaption}

\newtheorem{thm}{Theorem}
\newtheorem*{theorem A}{Theorem A}
\newtheorem{cor}[thm]{Corollary}
\newtheorem{lem}{Lemma}

\theoremstyle{definition}

\theoremstyle{remark}
\newtheorem{rem}{Remark}
\numberwithin{equation}{section}
\newtheorem*{thm A}{Theorem A}

\begin{document}
\thispagestyle{empty}\pagestyle{myheadings}%
\markboth{\small $(p,q)-$ Meyer-K\"{o}nig-Zeller
Durrmeyer operators }{\small  Sharma H, Gupta C, Maurya R}
\begin{center}
\textbf{\Large Approximation properties of $(p,q)-$Meyer-K\"{o}nig-Zeller
Durrmeyer operators} \vskip0.1in
\textbf{Honey SharmaA$^{1}$, Cheena Gupta$^{2}$, Ramapati Maurya$^{3}$ }\\[0pt]
$^{1}${pro.sharma.h@gmail.com},\\
{Department of Mathematics, Gulzar Group of Institutes,\\[0pt]
Ludhiana, (Punjab) India}\\[0pt]
$^{2}${guptacheena21@gmail.com}\\
{I K G Punjab Technical University,\\
Kapurthala, (Punjab) India}\\
$^{3}${ramapatimaurya@gmail.com}\\
{I K G Punjab Technical University,\\
	Kapurthala, (Punjab) India}\\
{Department of Mathematics, Manav Rachna University, \\
Faridabad, (Haryana) India}

\end{center}
\section{Abstract}
In this paper, we introduce Durrmeyer type modification of Meyer-K\"{o}nig-Zeller operators based on $(p,q)-$integers. Rate of convergence of these operators are explored with the help of Korovkin type theorems. We establish some direct results for proposed operators. We also obtain statistical approximation properties of operators. In last section, we show rate of convergence of $(p,q)-$Meyer-K\"{o}nig-Zeller Durrmeyer operators for some functions by means of Matlab programming. \\[2pt]
\textbf{Keywords:} $q-$Calculus, $(p, q)-$Calculus, $(p, q)-$Meyer-K\"{o}nig-Zeller operator, Durrmeyer type operators, Modulus of continuity, Peetre $K-$functional, Statistical convergence.\\[2pt]
\textbf{Mathematical subject classification:} 41A25, 41A35.\\
\section{Introduction and Preliminaries}
Recently, Mursaleen et al. \cite{MM1}  introduced $(p, q)-$ analogue of Bernstein type operators. In the sequence, many researchers gave the $(p, q)-$analogue of various well known positive linear operators and study their approximation properties, for details one may refer to \cite{TA, VGA, KAD, HS2, HS1}. Now, We begin by recalling certain notations of $(p,q)$ calculus.\\ Let $0 < q < p \leq 1 $.
 The  $(p,q)$- integer is defined as
\begin{eqnarray*}
[n]_{p,q} = \frac{p^n-q^n}{p-q},  \quad n = 1,2.....
\end{eqnarray*}
 and the $(p,q)$-factorial is given by
\begin{align*}
[n]_{p,q}! &=
  \begin{cases}
   [1]_{p,q} [2]_{p,q}......... [n]_{p,q},        & \text n \geq 1 \\
   1,                                    & \text{n=0}.
  \end{cases}
\end{align*}

For integers  $0\leq k\leq n$, the $(p,q)$-binomial coefficient is defined as
\begin{equation*}
\left [\begin{array}{c}
                    n \\
                    k
                  \end{array}
\right]_{p,q} = \frac{[n]_{p,q}!}{[k]_{p,q}![n-k]_{p,q}!}.
\end{equation*}

Further, $(p,q)-$binomial function is expressed as
\begin{equation*}
(x+y)_{p,q}^n = \prod_{j=0}^{n-1}(p^{j} x + q^{j}y).
\end{equation*}

Recently, Sharma \cite{HS2} introduces the $(p,q)-$Beta function for $s,t \in \Re^+$ as
\begin{equation*}
\beta_{p,q}(t,s)= \int_0^1 x^{t-1} (1-qx)_{p,q}^{s-1} d_{p,q}x
\end{equation*}
and also obtain the relation between $(p, q)-$Beta function and $q-$Beta function as
\begin{equation*}
\beta_{p,q}(t,s)= p^{(s-1)(s-2)/2 -(t-1)} \beta_\frac{q}{p}(t,s),
\end{equation*}
where, $\beta_\frac{q}{p}(t,s)$ is $\frac{q}{p}-$ analogue of beta function. Using $\beta_q (t,s)= \frac{[t-1]!_q[s-1]!_q}{[s+t-1]!_q}$ and $[n]!_\frac{q}{p}=p^{-n(n-1)/2}[n]!_{p,q}$, we can write
\begin{equation}\label{eq:AA}
\beta_{p,q}(t,s)= p^{((s+t-1)(s+t-2)-(t-1)(t-2))/2-t+1}\frac{[t-1]!_{p,q}[s-1]!_{p,q}}{[s+t-1]!_{p,q}}.
\end{equation}

For $p = 1$, all the  notations of $(p,q)-$ calculus are reduced to $q-$calculus and further details on $(p,q)-$ calculus can be found in \cite{RC, PNS, VSA}. .\\

 In a recent studies, Kadak et al. \cite{KD} introduced a $(p,q)-$analogue of Meyer-K\"{o}nig-Zeller operators, for $0<q<p\leq1$, on a function defined on $[0,1]$ as
\begin{equation*}
 M_{n,p,q}(f;x)= \frac{1}{p^\frac{n(n+1)}{2}} \sum_{k=0}^\infty\begin{bmatrix}{n+k}\\{k}\end{bmatrix}_{p,q}x^k p^{-kn} (1-x)_{p,q}^{n+1}f \left(\frac{p^n[k]_{p,q}}{[n+k]_{p,q}}\right), \quad x\in[0,1)
\end{equation*}
and for $x=1,$ $M_{n,p,q}(f;1)= f(1)$.

Further, the moment of the operators are given in the following Lemma.
\begin{lem} (\cite{KD})\label{lem:A}
 For all $x\in[0,1]$ and $ 0<q<p\leq1$,  we have
\begin{eqnarray*}
M_{n,p,q}(1;x) &=& 1, \\
M_{n,p,q}(t;x) &=& x,  \\
x^2 \leq M_{n,p,q}(t^2;x) &\leq& \frac{p^n}{[n+1]_{p,q}}x +x^2.
\end{eqnarray*}
\end{lem}
In the past two decades, Studies of Durrmeyer variants of various operators remained the centre  of attraction for the researchers, for which one may refer to \cite{AAR, V1, V2, HS3, HS4}. Motivated by these studies, now we introduce the Meyer-K\"{o}nig-Zeller Durrmeyer operators based on $(p,q)-$integers in the following section.\\
\section{Construction of operator and Moment estimate}

For $0<q<p\leq1$ and function $f$ defined on $[0,1]$, the $(p,q)-$ Meyer-K\"{o}nig-Zeller Durrmeyer operators are defined as follows:
\begin{equation*}
\widetilde{M}_{n,k}^{(p,q)}(f;x)= \frac{[n+1]_{p,q}}{p^n}\sum^{\infty}_{k=0}m_{n,k}^{(p,q)}(x) (pq)^-k\int_0^1
b_{n,k}^{(p,q)}(qt) f(t) d_{p,q}t, \text{        } 0\leq x<1,
\end{equation*}
here,
\begin{equation*}
m_{n,k}^{(p,q)}(x)= \frac{1}{p^{kn+n(n+1)/2}}\begin{bmatrix}{n+k}\\{k}\end{bmatrix}_{p,q}x^k (1-x)_{p,q}^{n+1},
\end{equation*}

\begin{equation*}
b_{n,k}^{(p,q)}(qt)= \frac{1}{p^{k(n-1)+n(n-1)/2}}\begin{bmatrix}{n+k+1}\\{k}\end{bmatrix}_{p,q}(qt)^k (1-qt)_{p,q}^n
\end{equation*}
and $\widetilde{M}_{n,k}^{(p,q)}(f;1)=1$. Before computing the moments of $(p,q)-$Meyer-K\"{o}nig-Zeller Durrmeyer operators, we prove some lemmas as follows:
\begin{lem}\label{lem:1}
 Let $ 0<q<p\leq1$ and $s=0,1,2...$, we have
\begin{equation*}
\int_0^1b_{n,k}^{(p,q)}(qt) t^s d_{p,q}t = \frac{[n+k+1]_{p,q}![k+s]_{p,q}!}{[k]_{p,q}![n+k+s+1]_{p,q}!}  \frac{(pq)^k}{[n+1]_{p,q}} p^{n(s+1)}
\end{equation*}
\end{lem}
\begin{proof}
Lemma can be proved directly by using definition of $(p, q)-$beta operator and Equation (\ref{eq:AA}).
\end{proof}
\begin{lem}\label{lem:2}
For $r=1,2\ldots$ and $n>r$, we have
\begin{equation*}
\sum\limits_{k=0}^{\infty}\left[\begin {array}{c}
        n+k \\
        k
      \end{array}
\right]_{p,q}\frac{x^{k}(1-x)_{p,q}^{n+1}}{[n+k]_{p,q}^{\underline{r}}}p^{(r-n)k} =
\frac{\prod\limits_{j=0}^{r-1}(p^{n-j}-q^{n-j}x)}{[n]_{p,q}^{\underline{r}}} p^{\frac{(n-r)(n-r+1)}{2}}
\end{equation*}
where $[n]_{p,q}^{\underline{r}}=[n]_{p,q}[n-1]_{p,q}[n-2]_{p,q}\ldots[n-r+1]_{p,q}$.
\end{lem}
\begin{lem}\label{lem:3}
The identity
 \begin{equation*}
 \frac{1}{[n+k+r]_{p,q}} \leq \frac{1}{q^r[n+k]_{p,q}}, r\geq0
 \end{equation*}
 holds.
 \end{lem}

\begin{thm}\label{thm:1}
For all $x\in[0,1]$, $n\in\mathbb{N}$ and $0<q<p\leq1$, we have
\begin{eqnarray*}
\widetilde{M}_{n,k}^{(p,q)}\left(e_{0};x\right)&=&1,\\
\frac{x}{q^2}\left(1-\frac{q+1}{[n]_{p,q}}\right) \leq \widetilde{M}_{n,k}^{(p,q)}\left(e_{1};x\right)&\leq& \frac{x}{q} +\frac{(p^n-q^nx)}{q^2[n]_{p,q}},\\ \widetilde{M}_{n,k}^{(p,q)}\left(e_{2};x\right)&\leq&\frac{x^{2}}{q^{2}}+\frac{(p+q)^{2}}{q^{5}}\frac{(p^n-q^nx)}{[n]_{p,q}}x+\frac{p(p+q)}{q^{6}}\frac{(p^n-q^nx)(p^{n-1}-q^{n-1}x)}{[n]_{p,q}[n-1]_{p,q}},
\end{eqnarray*}
here, $e_i=t^i$ for $i=0,1,2$.
\end{thm}
\begin{proof}
First moment can be directly computed. We use the moments obtained for  $(p,q)-$Meyer-K\"{o}nig-Zeller operators in Lemma \ref{lem:A} to estimate moments of proposed Durrmeyer operators. By using Lemma \ref{lem:1} for $s=1$ and Lemma \ref{lem:2}, we get lower bound of second moment as follows:
\begin{eqnarray*}
\widetilde{M}_{n,k}^{(p,q)}(e_{1},x)&=&\sum_{k=0}^\infty m_{n,k}^{(p,q)}(x) \frac{[n+k+1]_{p,q}![k+1]_{p,q}!}{[k]_{p,q}![n+k+2]_{p,q}!}p^n \\
 &=& p^n\sum_{k=0}^\infty\frac{1}{p^{kn+n(n+1)/2}}\begin{bmatrix}{n+k}\\{k}\end{bmatrix}_{p,q}x^k (1-x)_{p,q}^{n+1} \frac{[k+1]_{p,q}}{[n+k+2]_{p,q}} \\
 &=& p^n\sum_{k=1}^\infty\frac{1}{p^{kn+n(n+1)/2}}\begin{bmatrix}{n+k-1}\\{k-1}\end{bmatrix}_{p,q}x^k (1-x)_{p,q}^{n+1}\frac{[k+1]_{p,q}}{[k]_{p,q}} \frac{[n+k]_{p,q}}{[n+k+2]_{p,q}}\\
&\geq& p^n\sum_{k=0}^\infty\frac{1}{p^{(k+1)n+n(n+1)/2}}\begin{bmatrix}{n+k}\\{k}\end{bmatrix}_{p,q}x^{k+1} (1-x)_{p,q}^{n+1}\frac{[n+k+1]_{p,q}}{[n+k+3]_{p,q}}\\
&=& \frac{x}{q}\sum_{k=0}^\infty m_{n,k}^{(p,q)}(x) \left(\frac{[n+k+2]_{p,q}-p^{n+k+1}}{[n+k+3]_{p,q}}\right)\\
&\geq& \frac{x}{q}\sum_{k=0}^\infty m_{n,k}^{(p,q)}(x) \left(\frac{[n+k+2]_{p,q}}{[n+k+3]_{p,q}}-\frac{1}{[n]_{p,q}}\right)\\
&=& \frac{x}{q}\sum_{k=0}^\infty m_{n,k}^{(p,q)}(x) \frac{[n+k+2]_{p,q}}{[n+k+3]_{p,q}}-\frac{x}{[n]_{p,q}q}\\
&=& \frac{x}{q^2}\sum_{k=0}^\infty m_{n,k}^{(p,q)}(x) \left(\frac{[n+k+3]_{p,q}-p^{n+k+2}}{[n+k+3]_{p,q}}\right)-\frac{x}{[n]_{p,q}q}\\
&&\geq \frac{x}{q^2}\sum_{k=0}^\infty m_{n,k}^{(p,q)}(x) \left(1-\frac{1}{[n]_{p,q}}\right)-\frac{x}{[n]_{p,q}q}\\
&=& \frac{x}{q^2}\left(1-\frac{q+1}{[n]_{p,q}}\right).
\end{eqnarray*}
By using inequality of Lemma \ref{lem:3}, upper bound can be obtained as below:
 \begin{eqnarray*}
\widetilde{M}_{n,k}^{(p,q)}(e_{1},x) &=& p^n\sum_{k=0}^\infty m_{n,k}^{(p,q)}(x) \frac{[k+1]_{p,q}}{[n+k+2]_{p,q}} \\
 &\leq& p^n\sum_{k=0}^\infty m_{n,k}^{(p,q)}(x) \frac{p^k+q[k]_{p,q}}{q^2[n+k]_{p,q}} \\
 &=& \frac{p^n}{q^2}\sum_{k=0}^\infty m_{n,k}^{(p,q)}(x) \frac{p^k}{[n+k]_{p,q}} + \frac{1}{q}\sum_{k=0}^\infty m_{n,k}^{(p,q)}(x) \frac{p^n[k]_{p,q}}{[n+k]_{p,q}}
  \end{eqnarray*}
  \begin{eqnarray*}
&=& \frac{p^n}{q^2}\sum_{k=0}^\infty \frac{1}{p^{kn}+n(n+1)/2}\begin{bmatrix}{n+k}\\{k}\end{bmatrix}_{p,q}x^k (1-x)_{p,q}^{n+1} \frac{p^k}{[n+k]_{p,q}} + \frac{x}{q} \\
&=&  \frac{\left(p^n-q^nx\right)}{q^2[n]_{p,q}}+\frac{x}{q}. \\
\end{eqnarray*}
To estimate third moment, we use Lemma \ref{lem:1} for $s=2$ and Lemma \ref{lem:2} as follows:
\begin{eqnarray*}
\widetilde{M}_{n,k}^{(p,q)}(e_{2},x)&=&\sum_{k=0}^\infty m_{n,k}^{(p,q)}(x) \frac{[n+k+1]_{p,q}![k+2]_{p,q}!}{[k]_{p,q}![n+k+3]_{p,q}!}p^{2n} \\
&=& p^{2n}\sum_{k=0}^\infty m_{n,k}^{(p,q)}(x) \frac{[k+2]_{p,q}[k+1]_{p,q}}{[n+k+3]_{p,q}[n+k+2]_{p,q}} \\
&\leq& \frac{p^{2n}}{q^6} \sum_{k=0}^\infty m_{n,k}^{(p,q)}(x)\frac{(p+q)p^{2k}+ (p+2q)qp^k [k]_{p,q}+q^3[k]_{p,q}^2 }{[n+k]_{p,q}[n+k-1]_{p,q}}\\
&=& \frac{p^{\frac{-n(n-3)}{2}}}{q^6}(p+q)\sum_{k=0}^\infty  p^{-kn}\begin{bmatrix}{n+k}\\{k}\end{bmatrix}_{p,q} \frac{x^k (1-x)_{p,q}^{n+1}}{[n+k]_{p,q}[n+k-1]_{p,q}}p^{2k}\\
&+& \frac{p^{\frac{-n(n-3)}{2}}}{q^5}(p+2q)\sum_{k=0}^\infty  p^{-kn}\begin{bmatrix}{n+k}\\{k}\end{bmatrix}_{p,q} \frac{x^k (1-x)_{p,q}^{n+1}}{[n+k]_{p,q}[n+k-1]_{p,q}}p^{k}[k]_{p,q}\\
&+& \frac{p^{\frac{-n(n-3)}{2}}}{q^3}\sum_{k=0}^\infty  p^{-kn}\begin{bmatrix}{n+k}\\{k}\end{bmatrix}_{p,q} \frac{x^k (1-x)_{p,q}^{n+1}}{[n+k]_{p,q}[n+k-1]_{p,q}}[k]_{p,q}^2\\
&=& I_1+I_2+I_3.
\end{eqnarray*}
By using lemma \ref{lem:3}, $I_1$ can be obtained as
\begin{equation*}
I_1=\frac{p(p+q)}{q^6} \frac{\left(p^n-q^nx\right)\left(p^{n-1}-q^{n-1}x\right)}{[n]_{p,q}[n-1]_{p,q}}.
\end{equation*}
Computations for $I_2$ are as follows:
\begin{eqnarray*}
I_2&=&\frac{p^{\frac{-n(n-3)}{2}}}{q^5}(p+2q)\sum_{k=1}^\infty p^{-kn}\begin{bmatrix}{n+k-1}\\{k-1}\end{bmatrix}_{p,q} \frac{x^k (1-x)_{p,q}^{n+1}}{[n+k-1]_{p,q}}p^{k}\\
&=& \frac{p^{\frac{-n(n-1)}{2}+1}}{q^5}(p+2q)x\sum_{k=0}^\infty p^{-kn} \begin{bmatrix}{n+k}\\{k}\end{bmatrix}_{p,q} \frac{x^k(1-x)_{p,q}^{n+1}}{[n+k]_{p,q}}p^{k}\\
&=&\frac{p(p+2q)}{q^5[n]_{p,q}}\left(p^n-q^nx\right)x.
\end{eqnarray*}
$I_3$ can be obtained as follows:
\begin{eqnarray*}
I_3&=& \frac{p^{\frac{-n(n-3)}{2}}}{q^3}\sum_{k=1}^\infty p^{-kn}\begin{bmatrix}{n+k-1}\\{k-1}\end{bmatrix}_{p,q} \frac{x^k (1-x)_{p,q}^{n+1}}{[n+k-1]_{p,q}}[k]_{p,q}\\
&=&\frac{p^{\frac{-n(n-1)}{2}}}{q^3}x \sum_{k=0}^\infty p^{-kn} \begin{bmatrix}{n+k}\\{k}\end{bmatrix}_{p,q}  \frac{x^k (1-x)_{p,q}^{n+1}}{[n+k]_{p,q}}[k+1]_{p,q}\\
&=&\frac{p^{\frac{-n(n-1)}{2}}}{q^3}x \sum_{k=0}^\infty p^{-kn} \begin{bmatrix}{n+k}\\{k}\end{bmatrix}_{p,q}  \frac{x^k (1-x)_{p,q}^{n+1}}{[n+k]_{p,q}}p^k
\end{eqnarray*}
\begin{eqnarray*}
&+& \frac{p^{\frac{-n(n-1)}{2}}}{q^3}x \sum_{k=0}^\infty p^{-kn} \begin{bmatrix}{n+k}\\{k}\end{bmatrix}_{p,q}  \frac{x^k (1-x)_{p,q}^{n+1}}{[n+k]_{p,q}}q[k]_{p,q}\\
&=& \frac{x^2}{q^2} + \left(\frac{\left(p^n-q^nx\right)}{q^3[n]_{p,q}} \right)x.
\end{eqnarray*}
By using $I_1$, $I_2$ and $I_3$, we get upper bound of second moment.
 \end{proof}
 \begin{cor}\label{cor:1}
Central moments of operators are
\begin{eqnarray*}
\widetilde M_{n,k}^{(p,q)}\left(\psi_{1};x\right)&\leq& \frac{p^n-q^nx}{q^2[n]_{p,q}}+ \left(\frac{1}{q}-1\right)x,\\
\widetilde M_{n,k}^{(p,q)}\left(\psi_{2};x\right)&\leq& x^2\left(1-\frac{1}{q^2}+\frac{2(q+1)}{q^{2}[n]_{p,q}}\right)+\frac{(p+q)^{2}}{q^{5}}\frac{(p^n-q^nx)}{[n]_{p,q}}x\\ &+&\frac{p(p+q)}{q^{6}}\frac{(p^n-q^nx)(p^{n-1}-q^{n-1}x)}{[n]_{p,q}[n-1]_{p,q}}.
\end{eqnarray*}
where $\psi_{i}(x)=(t-x)^{i}$ for $i=1,2$.
\end{cor}
\begin{proof}
 By the linearity of $\widetilde M_{n,k}^{(p,q)}$ and Theorem \ref{thm:1}, central moments can be obtained directly.
 \end{proof}
  \begin{rem}\label{rem:1}
 For $0<q<p\leq 1$, by simple computations $\lim_{n\rightarrow\infty} [n]_{p,q}=1/(p-q)$. In order to obtain results for order of convergence of the operator, we take $q_n\in(0,1)$, $p_n \in (q_n,1]$ such that $\lim_{n\rightarrow\infty} p_n=\lim_{n\rightarrow\infty} q_n=1$, $\lim_{n\rightarrow\infty} p^n_n=a$ and  $\lim_{n\rightarrow\infty} q^n_n=b$, so that $\lim_{n\rightarrow\infty} \frac{1}{[n]_{p_n,q_n}}=0$. Such a sequence can always be constructed for example, we can take $q_n=1-1/2n$ and $p_n=1-1/3n$, clearly $\lim_{n\rightarrow\infty}p_n^n=e^{-1/3}$, $\lim_{n\rightarrow\infty}q_n^n=e^{-1/2}$ and $\lim_{n\rightarrow\infty} \frac{1}{[n]_{p_n,q_n}}=0$.
  \end{rem}
 \section{Rate of convergence}
 We denote $W^2 =\{g\in C[0, 1]:g',g''\in C[0, 1]\}$. For $\delta>0,$~~~$K-$functional is defined as
\begin{equation*}
K_2(f,\delta)=\inf_{g\in W^2}\{\|f-g\|+\delta\|g''\|\},
\end{equation*}
 here norm $||.||$ denotes the supremum norm on $C[0, 1]$. Following the well-known inequality given in  DeVore and Lorentz \cite{RAG}, there exists an absolute constant $ C>0 $  such that
\begin{equation*}
K_2(f,\delta)\leq C\omega_2(f,\sqrt{\delta}),
\end{equation*}
here, $\omega_2(f,\sqrt{\delta})$ is second order modulus of continuity for $f\in C[0, 1]$, defined as
\begin{equation*}
\omega_2(f,\sqrt{\delta})=\sup_{0<h\leq\sqrt{\delta}}\sup_{x, x+h \in [0,1]}|f(x+2h)-2f(x+h)+f(x)|.
\end{equation*}
By
$\omega(f,\delta)=\sup_{0<h\leq\delta}\sup_{x, x+h \in [0,1]}|f(x+h)-f(x)|$, we denote the usual modulus of continuity for $f\in C[0, 1]$.
   \begin{thm}
 Let $(p_n)_n$ and $(q_n)_n$ be the sequence as defined in Remark \ref{rem:1}. Then for each $f \in C[0,1]$,
 $\widetilde{M}_{n,k}^{(p_n,q_n)}(f;x)$ converges uniformly to $f$.
 \end{thm}
 \begin{proof}
By Korovkin theorem, it is sufficient to show that $\lim_{n\rightarrow\infty} \|\widetilde{M}_{n,k}^{(p_n,q_n)}(t^m;x) -x^m\| = 0 $ for $ m=0,1,2$. For $m=0$ results hold trivially. Using Theorem \ref{thm:1}, we obtain the results for $m=1, 2$ as follows:
\begin{eqnarray*}
\lim_{n\rightarrow\infty} \|\widetilde{M}_{n,k}^{(p_n,q_n)} (t;x) -x\| &\leq& \lim_{n\rightarrow\infty} \left|\frac{x}{q_n} +\frac{(p_n^n-q_n^nx)}{q_n^2[n]_{p_n,q_n}} -x\right|\\
&\leq& \lim_{n\rightarrow\infty} \left|\frac{p_n^n}{q_n^2[n]_{p_n,q_n}}\right| + \lim_{n\rightarrow\infty} \left|\frac{1}{q_n}-\frac{q_n^{(n-2)}}{[n]_{p_n,q_n}}-1\right|x\\
&=& 0.
\end{eqnarray*}
Finally,
\begin{eqnarray*}
\lim_{n\rightarrow\infty} \|\widetilde{M}_{n,k}^{(p_n,q_n)} (t^2;x) -x^2\|
&\leq& \lim_{n\rightarrow\infty}\left|\frac{1}{q_n^{2}}-1\right|x^2 +\lim_{n\rightarrow\infty} \left|\frac{(p_n+q_n)^{2}}{q_n^{5}}\frac{(p_n^n-q_n^nx)}{[n]_{p_n,q_n}}\right|x\\
&+& \lim_{n\rightarrow\infty} \left|\frac{p_n(p_n+q_n)}{q_n^{6}}\frac{(p_n^n-q_n^nx)(p_n^{n-1}-q_n^{n-1}x)}{[n]_{p_n,q_n}[n-1]_{p_n,q_n}}\right|\\
&=& 0.
\end{eqnarray*}
Hence the Theorem.
\end{proof}
\begin{thm}
Let $(p_n)_n$ and $(q_n)_n$ be the sequence as defined in Remark \ref{rem:1}. Let $f\in C[0,1]$. Then for all $n\in N$, there exists an absolute constant $C>0$ such that
\begin{equation*}
|\widetilde{M}_{n,k}^{(p_n,q_n)}(f;x) - f(x)| \leq C\omega_2(f,\delta_n(x)) + \omega(f,\alpha_n(x)),
\end{equation*}
here,
\begin{equation*}
\delta_n(x) = \left\{\widetilde{M}_{n,k}^{(p_n,q_n)}((t-x)^2;x) +\widetilde{M}_{n,k}^{(p_n,q_n)}(t-x;x) \right\}^\frac{1}{2}
\end{equation*}
and
\begin{equation*}
\alpha_n(x) = \widetilde{M}_{n,k}^{(p_n,q_n)}(t-x;x).
\end{equation*}
\end{thm}
\begin{proof}
For $x \in [0,1],$ we consider the operators $ M_n^*(f;x)$ as
\begin{equation*}
M_n^*(f;x) = \widetilde{M}_{n,k}^{(p_n,q_n)}(f;x) + f(x)- f\left(\frac{x}{q_n}+ \frac{p_n^n-q_n^nx}{q_n^2[n]_{p_n,q_n}}\right).
\end{equation*}
Using first central moment of $\widetilde{M}_{n,k}^{(p_n,q_n)}$ and positivity of operator, we immediately get $M_n^*(t-x;x)=0$.\\
For $g \in W^2$ and $x\in[0,1]$, using the Taylor's formula
\begin{equation*}
g(t) = g(x) + g'(x)(t-x) + \int_x^t (t-u) g''(u) du.
\end{equation*}
Therefore,
\begin{eqnarray*}
M_n^*(g;x)-g(x) &=& g'(x)M_n^*((t-x);x) + M_n^* \left(\int_x^t (t-u) g''(u) du;x \right) \\
                  &=&  \widetilde{M}_{n,k}^{(p_n,q_n)} \left(\int_x^t (t-u) g''(u) du;x \right)\\
                  &-& \int_x^ {\frac{x}{q_n}+ \frac{p_n^n-q_n^nx}{q_n^2[n]_{p_n,q_n}}} \left(\frac{x}{q_n}+ \frac{p_n^n-q_n^nx}{q_n^2[n]_{p_n,q_n}} - u\right)g''(u)du. \\
\end{eqnarray*}
Finally, we have
\begin{eqnarray*}
|M_n^*(g;x)-g(x)| &\leq& \left|\widetilde{M}_{n,k}^{(p_n,q_n)} \left(\int_x^t (t-u) g''(u) du;x \right)\right|\\
&+& \left|\int_x^ {\frac{x}{q_n}+ \frac{p_n^n-q_n^nx}{q_n^2[n]_{p_n,q_n}}} \left(\frac{x}{q_n}+ \frac{p_n^n-q_n^nx}{q_n^2[n]_{p_n,q_n}} - u\right)g''(u)du\right|\\
&\leq& \|g''\|\widetilde{M}_{n,k}^{(p_n,q_n)}((t-x)^2;x) + \left(\frac{x}{q_n}+ \frac{p_n^n-q_n^nx}{q_n^2[n]_{p_n,q_n}} - x\right)^2 \|g''\|\\
&=& \delta_n^2(x)\|g''\|.
\end{eqnarray*}
Also, we have
\begin{equation*}
|M_n^*(f;x)| \leq |\widetilde{M}_{n,k}^{(p_n,q_n)}(f;x)| +2\|f\| \leq3\|f\|.
\end{equation*}
Therefore,
\begin{eqnarray*}
|\widetilde{M}_{n,k}^{(p_n,q_n)}(f;x)-f(x)| &\leq& |M_n^*(f-g;x)-(f-g)(x)| + \left|f\left(\frac{x}{q_n}+ \frac{p_n^n-q_n^nx}{q_n^2[n]_{p_n,q_n}}\right) -f(x)\right|\\&+&|M_n^*(g;x)-g(x)|\\
&\leq& |M_n^*(f-g;x)| + |(f-g)(x)|+\left|f\left(\frac{x}{q_n}+ \frac{p_n^n-q_n^nx}{q_n^2[n]_{p_n,q_n}}\right) -f(x)\right|\\&+&|M_n^*(g;x)-g(x)|\\
&\leq& 4\|f-g\| + \omega\left(f,\alpha_n(x)\right) + \delta_n^2(x)\|g''\|.
\end{eqnarray*}
On taking the infimum on the right hand side over all $g\in W^2$  and by the definition of $K-$functional, we get
\begin{equation*}
|\widetilde{M}_{n,k}^{(p_n,q_n)}(f;x)-f(x)|\leq 4K_2(f,\delta_n^2(x)) +\omega(f,\alpha_n(x)).
\end{equation*}
\end{proof}
\section{Statistical Approximation}
In this section, by using a Bohman-Korovkin type theorem proved in
\cite{Gadjiev_Orhan}, we present the statistical approximation
properties of purposed operator.

At this moment, we recall the concept of statistical convergence.\\
A sequence $(x_{n})_{n}$ is said to be statistically convergent to
a number $L$, denoted by $st-\lim\limits_{n}x_{n}=L$ if, for every
$\varepsilon>0$,\\
    \begin{equation*}\delta\{n\in\mathbb{N}:|x_{n}-L|\geq\varepsilon\}=0,\end{equation*}
     where
    \begin{equation*}\delta(S):=\frac{1}{N}\sum_{k=1}^N \chi_S(j)\end{equation*}
is the natural density of set $S\subseteq \mathbb{N}$ and $\chi_S$
is the characteristic function of S.

Let $C_{B}(D)$ represents the space of all continuous functions on D
 and bounded on entire real line, where D is any interval on real
 line. It can be easily shown that $C_{B}(D)$ is a Banach space with
 supreme norm. Also $\widetilde{M}_{n,k}^{(p,q)}(f;x)$ are well defined for any
 $f \in C_{B}([0,1])$.

\begin{thm} \label{thm:111}(\cite{Gadjiev_Orhan})
 Let $\left(L_{n}\right)_{n}$ be a sequence
of positive linear operators from $C_{B}([a,b])$ into
$B([a,b])$ and satisfies the condition that
 \begin{equation*}
st-\lim\limits_{n}\left\|L_{n}e_{i}-e_{i}\right\|=0\textit{ for
all } i=0,1,2.
\end{equation*} Then,
\begin{equation*}
    st-\lim\limits_{n}\left\|L_{n}f-f\right\|=0 \textit{ for all } f\in C_{B}([a,b]).
\end{equation*}
\end{thm}

\begin{thm}
Let $\{p_n\}_n,\{q_n\}_n$ be sequences such that
\begin{equation*}
st-\lim_{n\to\infty}q_n=1,~~st-\lim_{n\to\infty}{q_n}^n=a,
\end{equation*}
\begin{equation*}
st-\lim_{n\to\infty}p_n=1,~~st-\lim_{n\to\infty}{p_n}^n=b.
\end{equation*} Then, $\widetilde{M}_{n,k}^{(p_n,q_n)}(f,x)$ converges statistically to $f$.
Therefore,
\begin{equation*}
st-\lim\limits_{n}\left\|\widetilde{M}_{n,k}^{(p_n,q_n)}(f(t);x)-f\right\|_{C[0,1]}=0 \textit{ for all } f\in C[0,1].
\end{equation*}
\end{thm}
\begin{proof}
By Theorem \ref{thm:111}, it is sufficient to prove that
\begin{equation*}
st-\lim\limits_{n}\left\|\widetilde{M}_{n,k}^{(p_n,q_n)}(f_i(t);x)-f_i(x)\right\|_{C[0,1]}=0 \textit{ for all } i=0,1,2.
\end{equation*}
Based on Theorem \ref{thm:1}, we have
\begin{equation*}
st-\lim\limits_{n}\left\|\widetilde{M}_{n,k}^{(p_n,q_n)}(1;x)-1\right\|_{C[0,1]}=0,
\end{equation*}
\begin{equation*}
|\widetilde{M}_{n,k}^{(p_n,q_n)}(t;x)-x| \leq \left|\frac{p_n^n}{q_n^2[n]_{p_n,q_n}} + \frac{x}{q_n}-\frac{xq_n^{(n-2)}}{[n]_{p_n,q_n}}-x\right|
\end{equation*}
and
\begin{equation*}
|\widetilde{M}_{n,k}^{(p_n,q_n)}(t^2;x)-x^2| \leq \left|\frac{x^2}{q_n^{2}}-x^2 + \frac{x(p_n+q_n)^{2}}{q_n^{5}}\frac{(p_n^n-q_n^nx)}{[n]_{p_n,q_n}}
+ \frac{p_n(p_n+q_n)}{q_n^{6}}\frac{(p_n^n-q_n^nx)(p_n^{n-1}-q_n^{n-1}x)}{[n]_{p_n,q_n}[n-1]_{p_n,q_n}}\right|.
\end{equation*}
 By taking supremum over $x\in[0,1]$ in above inequalities, we get
\begin{equation*}
 |\widetilde{M}_{n,k}^{(p_n,q_n)}(t;x)-x| \leq \left|\frac{p_n^n}{q_n^2[n]_{p_n,q_n}} + \frac{1}{q_n}-\frac{q_n^{(n-2)}}{[n]_{p_n,q_n}}-1\right|
 \end{equation*}
 and
 \begin{equation*}
 |\widetilde{M}_{n,k}^{(p_n,q_n)}(t^2;x)-x^2| \leq \left|\frac{1}{q_n^{2}}-1 + \frac{(p_n+q_n)^{2}}{q_n^{5}}\frac{(p_n^n-q_n^nx)}{[n]_{p_n,q_n}}
+ \frac{p_n(p_n+q_n)}{q_n^{6}}\frac{(p_n^n-q_n^n)(p_n^{n-1}-q_n^{n-1})}{[n]_{p_n,q_n}[n-1]_{p_n,q_n}}\right|.
\end{equation*}
By using fact that $st-\lim_{n}q_{n}=1$ and $st-\lim_{n}p_{n}=1$, we get
\begin{eqnarray*}
st-\lim\limits_{n}\left\|\widetilde{M}_{n,k}^{(p_n,q_n)}(t;x)-x\right\|_{C[0,1]}&=&0,\\ st-\lim\limits_{n}\left\|\widetilde{M}_{n,k}^{(p_n,q_n)}(t^2;x)-x^2\right\|_{C[0,1]}&=&0.
\end{eqnarray*}
Hence the theorem.
\end{proof}
In the next theorem, we  estimate the rate of convergence by using the concepts of modulus of continuity.
\begin{thm}
Let $\{p_n\}_n,\{q_n\}_n$ be sequences such that
\begin{equation*}
st-\lim_{n\to\infty}q_n=1,~~st-\lim_{n\to\infty}{q_n}^n=a,
\end{equation*}
\begin{equation*}
st-\lim_{n\to\infty}p_n=1,~~st-\lim_{n\to\infty}{p_n}^n=b.
\end{equation*}Then,
\begin{equation}
|\widetilde{M}_{n,k}^{(p_n,q_n)}(f;x)-f|\leq 2\omega(f,\surd\delta_{n})
\end{equation}
for all $f\in C[0,1]$, here $\delta_{n}=\widetilde{M}_{n,k}^{(p_n,q_n)}((t-x)^2;x)$.
\end{thm}
\begin{proof}
By the linearity and monotonicity of the operator, we get
\begin{equation*}
|\widetilde{M}_{n,k}^{(p_n,q_n)}(f;x)-f|\leq \widetilde{M}_{n,k}^{(p_n,q_n)}(|f(t)-f(x)|;x),
\end{equation*}
also, by property of modulus of continuity
\begin{equation*}
|f(t)-f(x)|\leq\left(1+\frac{(t-x)^2
}{\delta^{2}}\right)\omega(f,\delta).
\end{equation*}
By using above facts, we get
\begin{equation*}
|\widetilde{M}_{n,k}^{(p_n,q_n)}(f;x)-f| \leq
\left(\widetilde{M}_{n,k}^{(p_n,q_n)}\left(1;x\right)+\frac{1}{\delta^{2}}\widetilde{M}_{n,k}^{(p_n,q_n)}\left((t-x)^{2}
;x\right)\right)\omega(f,\delta).
\end{equation*}
So, letting $\delta_{n}=\widetilde{M}_{n,k}^{(p_n,q_n)}\left((t-x)^{2}
;x\right)$ and take $\delta=\surd\delta_{n}$, we finally get result.
\end{proof}
\section{Graphical Illustrations}
In this section, we show approximation by $(p, q)-$Meyer-K\"{o}nig-Zeller Kantrovich operators using Matlab programming for functions $f(x)=$ $(x-2/3)(x-4/5)$, $(x-1/4)(x-2/3)(x-4/5)$, $(x-1/3)(x-2/3)(x-3/5)(x-4/5)$ and $(x-1/3)(x-2/3)(x-3/5)(x-4/5)(x-5/7)$ taking $n=25$ and $k=150$.
\newpage
\begin{figure}\label{fig:hainx}
\centering\vspace{0 cm}
    \begin{subfigure}[b]{.55\textwidth}
        \includegraphics[width=\textwidth]{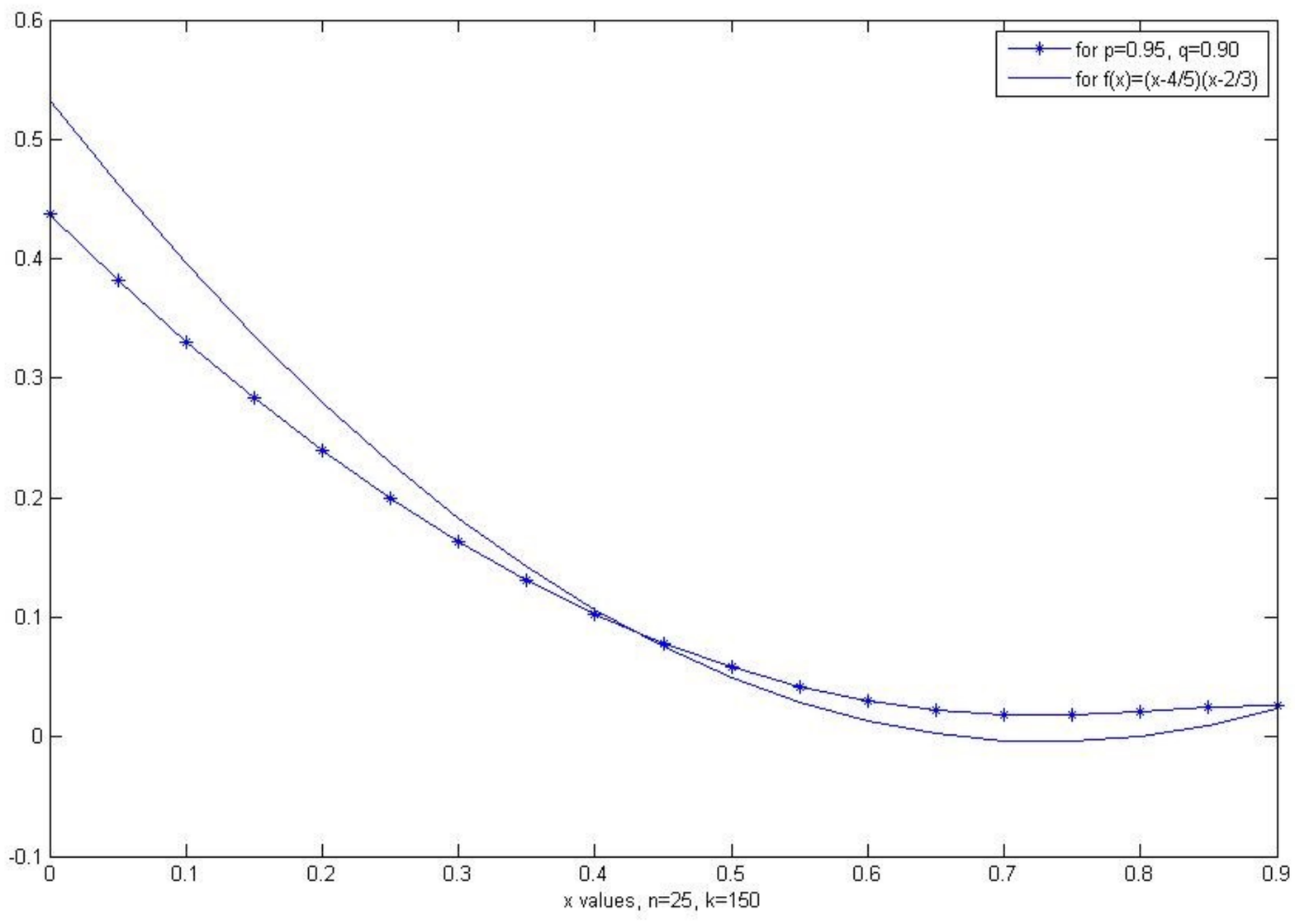}
        \vspace{-2\baselineskip}
        \caption{$f(x)=(x-2/3)(x-4/5)$}\label{fig:lab}
    \end{subfigure}
    ~\quad
    \begin{subfigure}[b]{.55\textwidth}
        \includegraphics[width=\textwidth]{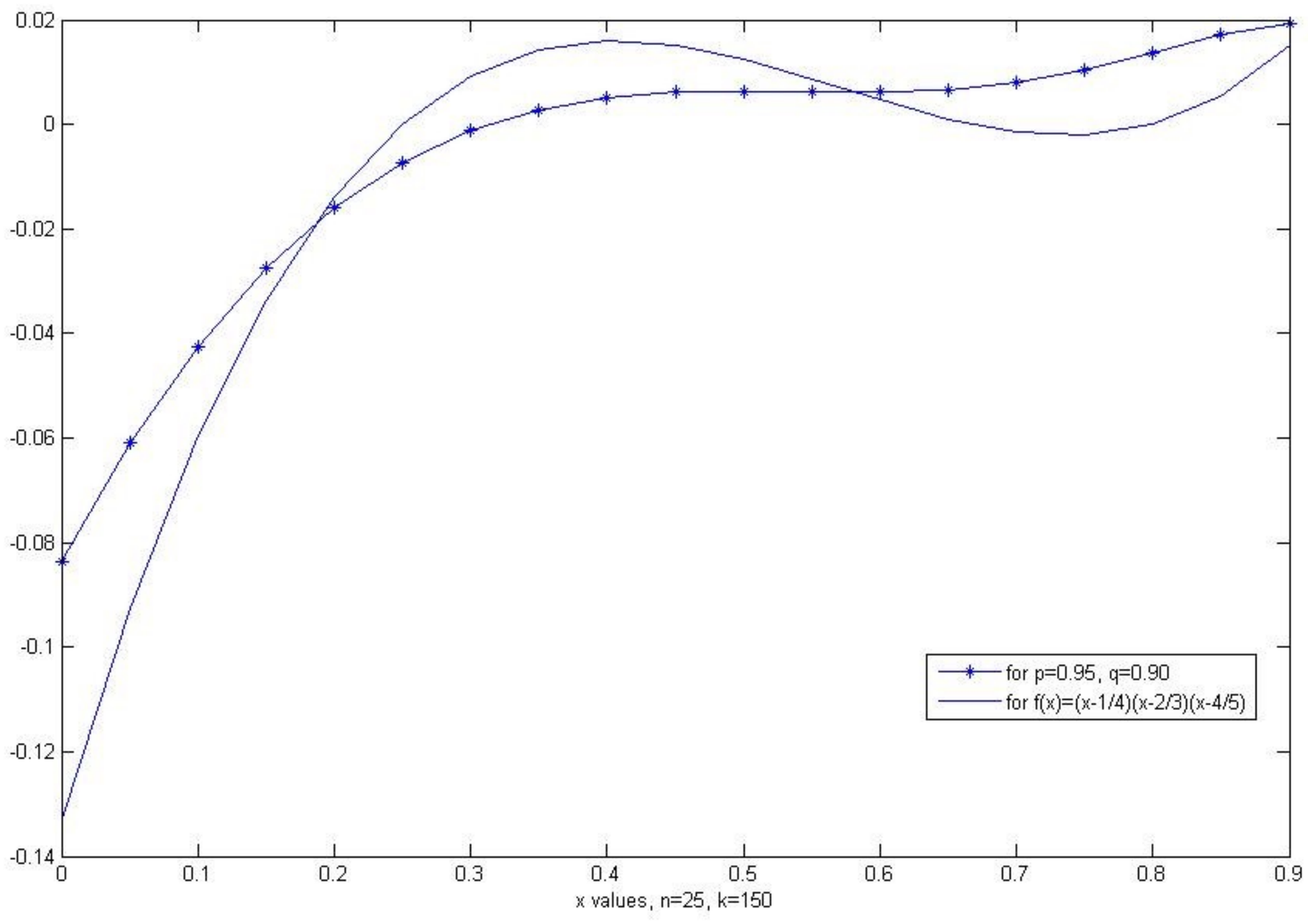}
        \vspace{-2\baselineskip}
        \caption{$f(x)=(x-1/4)(x-2/3)(x-4/5)$}\label{fig:lab1}
    \end{subfigure}\\[0ex]
    \begin{subfigure}[b]{.55\textwidth}
        \includegraphics[width=\textwidth]{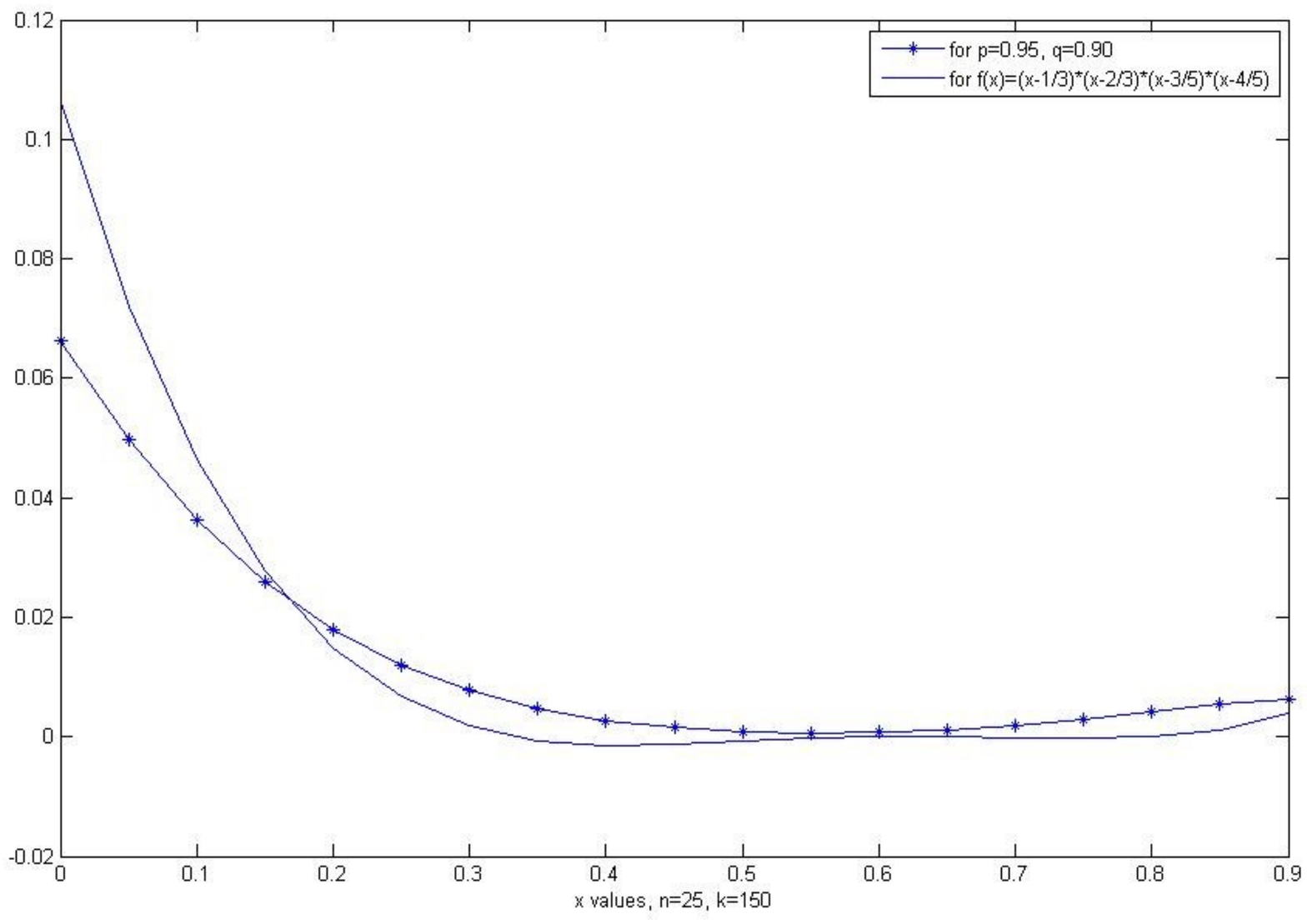}
        \vspace{-2\baselineskip}
        \caption{$f(x)=(x-1/3)(x-2/3)(x-3/5)(x-4/5)$}\label{fig:lab2}
   \end{subfigure}
   ~\quad
   \begin{subfigure}[b]{.55\textwidth}
        \includegraphics[width=\textwidth]{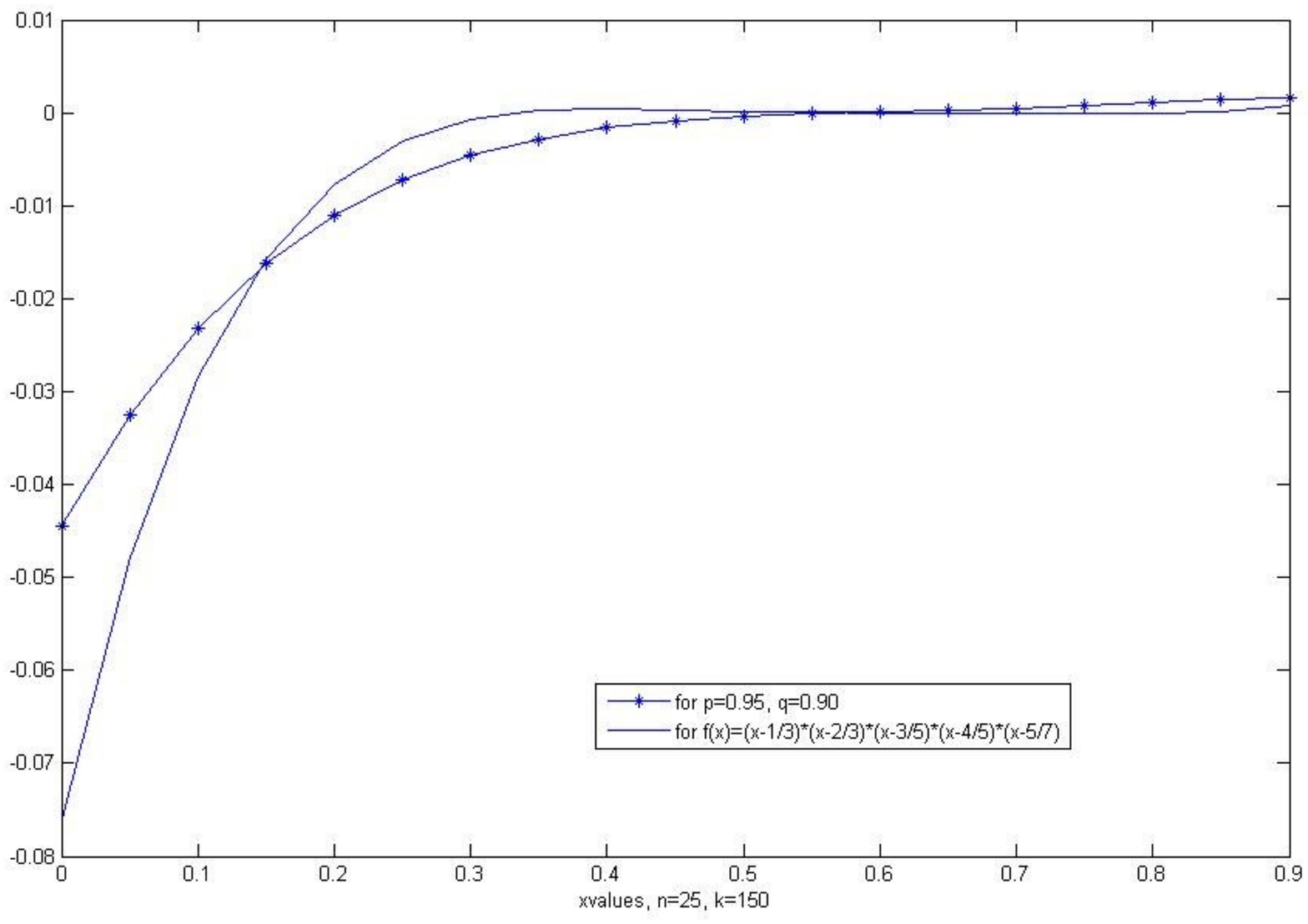}
        \vspace{-2\baselineskip}
        \caption{$f(x)=(x-1/3)(x-2/3)(x-3/5)(x-4/5)(x-5/7)$}\label{fig:lab3}
   \end{subfigure}
\end{figure}

\end{document}